\documentclass[10pt,twoside]{siamart1116}

\usepackage[english]{babel}
\usepackage{graphicx,epstopdf,epsfig}
\usepackage{amsfonts,epsfig,fancyhdr,graphics, hyperref,amsmath,amssymb}

\newcommand{\Names}{Huai-An Diao and Tong-Yu Zhou}
\newcommand{\Title}{Linearized estimate of the backward error for the equality constrained indefinite least squares problem}

\renewcommand{\theequation}{\thesection.\arabic{equation}}
\renewtheorem{theorem}{Theorem}[section]

\newcommand{\vect}{{\sf{vec}}}

\newcommand{\rt}{{\top }}

\def \ILSE {    {\tt {ILSE}  }  }

\newcommand {\R}        {{\mathbb R}}

\newcommand {\rank}     {\mathop{\rm rank}\nolimits}

\begin{document}

\bibliographystyle{plain}

\setcounter{page}{1}

\thispagestyle{empty}

 \title{\Title
 }

\author{
Huai-An Diao\thanks{School of Mathematics and Statistics,
Northeast Normal University, 
No. 5268 Renmin Street, Chang Chun 130024,
P.R. China.(hadiao@nenu.edu.cn). Supported by the Fundamental Research Funds for the Central Universities under grant 2412017FZ007.}
\and
Tong-Yu Zhou\thanks{School of Mathematics and Statistics,
Northeast Normal University, 
No. 5268 Renmin Street, Chang Chun 130024,
P.R. China. Current address: Shenyang No. 108 Middle School, 
Taiyuan North Road No. 14,
Shenyang  110001, P.R. China.(183051056@qq.com).}
}

\markboth{\Names}{\Title}

\maketitle

\begin{abstract}
In this note,
we concentrate on the backward error  of the  equality constrained indefinite least squares problem. For the normwise backward error  of the  equality constrained indefinite least square problem, we adopt the linearization method to derive the tight estimate for the exact backward normwise error. The numerical examples show that the linearization  estimate is effective for the normwise backward errors.
\end{abstract}





\section{Introduction}
\setcounter{equation}{0}

The indefinite least squares (ILS) problem \cite{5.0,1.0} is given by:
\begin{equation}\label{eq:ILS}
 \mbox{ILS}:\qquad\min_{x}(b-Ax)^\top\Sigma_{pq}(b-Ax),
\end{equation}
where  $A^\rt$ is the transpose of $A$,  $ A\in
\R^{m\times n},\, b\in
\R^{m},\, m\geq n $ and the signature matrix
\begin{equation}\label{eq:J}
\Sigma_{pq}=\left( \begin{matrix}
I_{p}&0\\0&-I_{q} \end{matrix} \right),\quad
 p+q=m.
\end{equation}
The ILS \eqref{eq:ILS} has applications in the total least squares problem \cite{30.0} and $H^\infty$-smoothing in  optimization  \cite{Hassibi96linearestimation,Sayed96inertiaproperties} see references and therein. The equality constrain indefinite linear least square problem (ILSE) was first proposed by Bojanczyk et al. in \cite{Higham2003ILSE}, which is a generalization of ILS. Suppose $ A\in \R^{m\times n},\,b\in \R^{m},\, B\in \R^{s\times n} ,\,d \in \R ^{s},\, m\geq n  $, and the signature matrix $\Sigma_{pq}$ is defined by \eqref{eq:J}. The ILSE has the form \begin{equation}\label{eq:ILSE}
 \mbox{ILSE}:\qquad \min_{x}(b-Ax)^{\top }\Sigma_{pq}(b-Ax)\quad \mbox{ subject to }\quad Bx=d.
\end{equation}
The existence  and uniqueness of the solution to ILSE is given in \cite{Higham2003ILSE}, i.e., 
  \begin{equation}\label{eq:condition}
  	\rank(B)=s, \,  x^\top (A^\top \Sigma_{pq}A)x>0,
  \end{equation}
where $x\in \mathcal{N}(B)$ and $\mathcal{N}(B)$  denotes the null space of $B$. The rank condition guarantees there exists a solution to the equality constrain in \eqref{eq:ILSE}, while the second one in \eqref{eq:condition}, which means that  $A^\top \Sigma_{pq}A$ is positive definite on ${\cal N}(B)$, ensures that the uniqueness of  a solution to the ILSE problem. When \eqref{eq:condition} is satisfied, the uniques solution $x$ to the ILSE problem \eqref{eq:ILSE} can be determined by the following normal equation 
  \begin{equation}\label{eq:normal}
  	 A^{\top }\Sigma_{pq}(b-Ax)=B^{\top }\xi, \, Bx=d,
  \end{equation}
  where $ \xi $ is a vector of Lagrange multipliers. On the other hand, the augmented system also defines the unique solution $x$ as follows
\begin{equation}\label{eq:aug}
{\cal A} {\bf x} :=\begin{bmatrix} 0& 0&B\cr
0&\Sigma_{pq}& A \cr
B^\rt & A^\rt& 0\end{bmatrix}\begin{bmatrix} \lambda\cr s\cr x\end{bmatrix}=\begin{bmatrix}
  d\cr b\cr 0
\end{bmatrix}:=\bf b,
\end{equation}
where $s=\Sigma_{pq}r$, $r$ is the residual vector $r=b-Ax$ and $\lambda=-\xi $. As pointed in \cite{Higham2003ILSE}, when \eqref{eq:condition} holds, the coefficient matrix $\cal A$ in \eqref{eq:aug} is invertible. For the numerical algorithms and theory for  ILSE, we refer to the papers \cite{MastronardiVanDooren2014BIT,Liu2010ILSEsolver,LiuWang2010ILSE,Mastronardi2015IMA} and etc.

Backward error analysis is important in numerical linear algebra, which can help us to examine the stability of numerical algorithms in matrix computation. Moreover, backward error can be used as the basis of effective stopping criteria for the iterative method for large scale problems. The concept of backward error can be traced to Wilkinson and others, see  \cite[Page 33]{Higham2002Book} for details. Many researchers had concentrated on the backward error analysis for the linear least squares problem \cite{Stewart,Higham,Walden,Karlson,Grcar03optimalsensitivity,Grcar2007}, the scale total least squares (STLS) problem \cite{Chang}, and the equality constrained least squares (LSE) problem and the least squares problem over a sphere (LSS)~\cite{Cox,Malyshev}. Since the formulae  and bounds for backward errors for least squares problems are expensive to evaluate, the linearization estimate for them was proposed; see  for \cite{Chang,Grcar03optimalsensitivity,HighamHigham1992BackCondStr,LiuXinguoNLAA2012} and references therein. To our best knowledge, there are no works on the normwise backward error for ILSE. In this paper, we will introduce the normwise backward error for ILSE and derive its linearization estimate.

The paper is organized as follows. We define the normwise backward error for ILSE and derive its linearization estimate  in Section \ref{sec:backward error}. We do some numerical examples to show the effectiveness of the proposed  linearization estimate for the normwise backward error in Section \ref{sec:nume ex}. At end, in Section \ref{sec:con},  concluding remarks are drawn.

\section{Main results}\label{sec:backward error}

In this section, we will focus on the linearization estimate for the normwise backward error for ILSE. Assume that we have the computed solution $y $  to  \eqref{eq:ILSE}. There exits matrices and vectors $E,\, F,\, f$  and $g $, which are the perturbations on $A, \, B,\, b$ and $d $ respectively, such that the computed solution $y $  is the exact solution of the following  perturbed ILSE problem
\begin{equation}\label{eq:pert1}
\min_{z}(b+f-(A+E)z)^{\top }\Sigma_{pq}(b+f-(A+E)z),\quad \mbox{subject to }\quad (B+F)y=d+g. 
\end{equation}
There may have many possible perturbations satisfying \eqref{eq:pert1}. Thus the following perturbation set 
$$
S _{\tt ILSE}(y)=\left\{(E,f,F,g)\, |\, (A+E)^{\top }\Sigma_{pq}(b+f-(A+E)y)=(B+F)^{\top }\xi,\, (B+F)y=d+g\right\},
$$
is introduced,  where $\xi$ is the vector given in \eqref{eq:normal}. Therefore the normwise backward error for  $y$ can be defined as follows:
\begin{equation}\label{eq:back}
\mu_{\ILSE}=\min\left\|\begin{bmatrix}
E & \theta_{1} f\\
\theta_{2} F & \theta_{3} g
\end{bmatrix}\right\|_{F} ,
\end{equation}
where $\|\cdot \|_F$ is Frobenius norm,  $(E,f,F,g)\in S_{\ILSE}(y)$ and $\theta_{1},\, \theta_{2}$ and  $ \theta_{3} $ are positive parameters to  give the weights to  $ f,\, F $ and $g$, respectively.  However it seems that it is difficult to derive the explicit expression of $\mu_{\ILSE}$ because of the non-linearity of \eqref{eq:pert1} with respect to the perturbations of $E,\, F,\, f$  and $g $.  In the following we will deduce the linearize estimate for $\mu_{\ILSE}$ via linearizing \eqref{eq:pert1} by dropping the higher order terms of the perturbations $E,\, F,\, f$  and $g $ in \eqref{eq:pert1}. 

First, we rewrite \eqref{eq:pert1} as follows: 
\begin{equation}\label{eq:pp}
J(\xi)\begin{bmatrix}
\vect(E)\\\theta_{1}f\\\theta_{2}\vect(F)\\\theta_{3}g
\end{bmatrix}=\begin{bmatrix}
B^{\top }\xi-A^{\top }\Sigma_{pq}r_y\\d-By
\end{bmatrix}-\begin{bmatrix}
E^\top \Sigma_{pq}\begin{bmatrix}
E&\theta_{1}f
\end{bmatrix}\begin{bmatrix}
-y\\\theta_{1}^{-1}
\end{bmatrix}\\0
\end{bmatrix},
\end{equation}
where $\vect(A)$ stacks the columns of $A$ one by one, 
\[J(\xi)=\begin{bmatrix}
I_{n}\otimes (r_y^{\top }\Sigma_{pq})-A^{\top }\Sigma_{pq}(y^{\top }\otimes I_{m})&\theta_{1}^{-1}A^{\top }\Sigma_{pq}&-\theta_{2}^{-1}(I_{n}\otimes \xi^{\top })&0\\0&0&\theta_{2}^{-1}(y^{\top }\otimes I_{s})& -\theta_{3}^{-1}I_{s}
\end{bmatrix},\] 
	the symbol $\otimes$ is Kronecker product,  $I_n$ denotes the $n\times n$ identity matrix and $ r_y=b-Ay $. Suppose  $r_y\neq0$, it is easy to verify that for any vector  $ \xi \in \R^{s} $, the matrix $ J(\xi) $ is full row rank.  Let 
$$ \tau(\xi)=\|J(\xi)^{\dagger}\|_{2},
\quad  \rho(\xi)=\left\|J(\xi)^{\dagger}\begin{bmatrix}
B^{\top }\xi-A^{\top }\Sigma_{pq}r_y\\d-By
\end{bmatrix}\right\|_{2},
	$$
where  $\|\cdot \|_2$ is the spectral norm of a matrix or 2-norm of a vector, and $A^\dagger$ is Moore-Penrose inverse of $A$. From the equation below
\begin{equation}\label{eq:qq}
\|J(\xi)^{\dagger}\|_{2}\leq\tau_{0}:=\left\|\begin{bmatrix}
I_{n}\otimes (r_y^{\top }\Sigma_{pq})-A^{\top }\Sigma_{pq}(y^{\top }\otimes I_{m})&\theta_{1}^{-1}A^\top \Sigma_{pq}&0\\0&0&-\theta_{3}^{-1}I_{s}
\end{bmatrix}^{\dagger}\right\|_{2},
\end{equation}
we know that   $\rho(\xi)$  is continuous with respect to $\xi$. We define the linearized estimate 
	\begin{equation}\label{eq:linear}
		\rho=\min_{\xi}\rho(\xi)
	\end{equation}
for $\mu_{\ILSE}$. In the following theorem, we prove that $\rho$ is an upper bound for $\mu_\ILSE $. 

\begin{theorem}\label{Th:uppbound}
If  $4\tau_{0}\rho\sqrt{\theta_{1}^{-2}+\|y\|_{2}^{2}}<1 $, we have $\mu_{\ILSE}<2\rho$.
\end{theorem}
\begin{proof}
Suppose $ \xi_{0}\in \R^{s}$  such that $\rho=\rho(\xi_{0})$. Consider the following nonlinear system: 
$$
J(\xi_{0})\begin{bmatrix}\vect(E)\\\theta_{1}f\\\theta_{2}\vect(F)\\\theta_{3}g\end{bmatrix}=\begin{bmatrix}B^{\top }\xi_{0}-A^{\top }\Sigma_{pq}r_y\\d-By\end{bmatrix}-\begin{bmatrix}
E^{\top }\Sigma_{pq}\begin{bmatrix}
E&\theta_{1}f
\end{bmatrix}\begin{bmatrix}
-y\\\theta_{1}^{-1}
\end{bmatrix}\\0
\end{bmatrix}.
$$
and the mapping $\Gamma:\, \R^{(n+1)(m+s)}\longrightarrow \R^{(n+1)(m+s)}$  defined by
	\begin{equation}\label{eq:map}
\Gamma\begin{bmatrix}\vect(E)\\\theta_{1}f\\\theta_{2}\vect(F)\\\theta_{3}g\end{bmatrix}
=J(\xi_{0})^{\dagger}\begin{bmatrix}B^{\top }\xi_{0}-A^{\top }\Sigma_{pq}r_y\cr d-By\end{bmatrix}-J(\xi_{0})^{\dagger}\begin{bmatrix}
E^{\top }\Sigma_{pq}\begin{bmatrix}
E&\theta_{1}f
\end{bmatrix}\begin{bmatrix}
-y\\\theta_{1}^{-1}
\end{bmatrix}\\0
\end{bmatrix}.
\end{equation}
From $J(\xi_{0})J(\xi_{0})^{\dagger}=I $, we know that any  fixed point of  $\Gamma$ is a solution to  \eqref{eq:map}. Let
		 $$
 \rho_{1}=\frac{2\rho}{1+\sqrt{1-4\sqrt{\theta_{1}^{-2}+\|y\|^{2}}\tau_{0}\rho}}, \quad   S_{2}=\left \{z\in \R^{(n+1)(m+s)}\mid \|z\|_2\leq\rho_{1}\right \},
 $$
then $ S_{2} $ is a convex and closed set of $\R^{(n+1)(m+s)}$. Moreover,  for arbitrary 
$ z =\begin{bmatrix}\vect(E)\\\theta_{1}f\\\theta_{2}\vect(F)\\\theta_{3}g\end{bmatrix}\in S_{2}$, we can deduce that
		 \[\|\Gamma z\|_{2}\leq \rho+\tau_{0}\sqrt{\theta_{1}^{-2}\|y\|_{2}^{2}}\|z\|_{2}^{2}\leq \rho_{1}, \]
which means that the continuous mapping  $\Gamma$ maps $ S_{2}$ to  $ S_{2}$. From Brouwer fixed point principle,  the mapping $\Gamma$ has a fixed point in $ S_{2}$, then we prove that
 \[\mu_{\ILSE}\leq\rho_{1}\leq2\rho.\]  \end{proof}
 \qed

In the following, we will consider how to estimate $\rho $, because it is not easy to derive the explicit expression for $\rho$. From \eqref{eq:qq}, we arrive at
		 \begin{equation}\label{eq:mm}
\rho(\xi)\leq \tau_{0}\left\|\begin{bmatrix}B^{\top }\xi_{0}-A^{\top }\Sigma_{pq}r_y\cr  d-By\end{bmatrix}\right\|_{2}.
\end{equation}
Apparently,   the minimal value of the upper bound in \eqref{eq:mm}  is attainable at 
		 \begin{equation}\label{eq:xi1}
		 	\xi_{1}=(B^{\top })^{\dagger}A^{\top }\Sigma_{pq}r_y.
		 \end{equation}
 We have $\rho\leq\rho(\xi_{1})$. From the above deduction, if  $4\tau_{0}\rho\sqrt{\theta_{1}^{-2}+\|y\|_{2}^{2}}<1$, it is not difficult to see that
		 	$$
		 	\mu_{\ILSE}<2\rho(\xi_{1}).
		 	$$

On the other hand, we  find the lower bound for  $\mu_{\ILSE}$ in the following theorem. 
 \begin{theorem}\label{Th:lower}
If $ r_y\neq0 $, then $$\mu_{\ILSE}\geq\frac{2\rho}{1+\sqrt{1+4\tau_{0}\sqrt{\theta_{1}^{-2}+\|y\|_{2}^{2}}\rho}}.$$
\end{theorem}
		 	
\begin{proof}
	For the following nonlinear system: \begin{equation}\label{eq:nn}
 J(\xi)^{\dagger}J(\xi)\begin{bmatrix}\vect(E)\\\theta_{1}f\\\theta_{2}\vect(F)\\\theta_{3}g\end{bmatrix}
 =J(\xi)^{\dagger}\begin{bmatrix}B^{\top }\xi_{0}-A^{\top }\Sigma_{pq}r_y\cr d-By\end{bmatrix}-
 J(\xi)^{\dagger}\begin{bmatrix}
E^{\top }\Sigma_{pq}\begin{bmatrix}
E&\theta_{1}f
\end{bmatrix}\begin{bmatrix}
-y\\\theta_{1}^{-1}
\end{bmatrix}\\0
\end{bmatrix},
\end{equation}
we know that any solution to \eqref{eq:pp} is also a solution to \eqref{eq:nn}. Because  $ J(\xi)$ is full row rank, any solution to \eqref{eq:nn} is a solution to \eqref{eq:pp}. From  $ \|J(\xi)^{\dagger}J(\xi)\|_2=1 $ and \eqref{eq:nn}, we can prove that
		 	\[ \rho(\xi)\leq\lambda+\tau_{0}\sqrt{\theta_{1}^{-2}+\|y\|_{2}^{2}}\lambda^{2},\]
where $\lambda=\left\|\begin{bmatrix}E&\theta_{1}f\\\theta_{2}F&\theta_{3}g\end{bmatrix}\right\|_{F} $. Then, we deduce that
\[\lambda\geq\frac{2\rho(\xi)}{1+\sqrt{1+4\tau_{0}\sqrt{\theta_{1}^{-2}+\|y\|_{2}^{2}}\rho(\xi)}}.\]
Because the function $ f(t)\equiv\left[2t/(1+\sqrt{1+4\tau_{0}\sqrt{\theta_{1}^{-2}+\|y\|_{2}^{2}}t}\right] $ is increasing with respect to  $t \,(t\geq 0)$. we prove this theorem. 
\end{proof}\qed 

Combing Theorems \ref{Th:uppbound} and \ref{Th:lower}, we have the following corollary.
\begin{corollary}
If $4\tau_{0}\rho\sqrt{\theta_{1}^{-2}+\|y\|_{2}^{2}}<1 $, then $\frac{2\rho}{1+\sqrt{2}}\leq\mu_{\ILSE}\leq2\rho$. 
\end{corollary}
The above corollary indicates that when   $\rho$ is small enough,  then $\rho$ is a good estimation for  $\mu_{\ILSE}$. On the contrary, the next result shows that if $\rho$ is not small, then $y$ cannot be a good approximate solution.

\begin{theorem}
Supoose $y$ is an approximation solution to \eqref{eq:ILSE}, and $x_{\ILSE}$ is its exact solution, then the following inequality 
$$\|x_{\ILSE}-y\|_2\geq \frac{1}{\left\|\begin{bmatrix}A^{\top }\Sigma_{pq}A\\B\end{bmatrix}\right\|_2 } \left\|\begin{bmatrix}B\xi_1-A^{\top }\Sigma_{pq}r_y\cr d-By\end{bmatrix}\right\|_2 $$
holds.
\end{theorem}

\begin{proof} Since $x_{\ILSE}$ is the exact solution to \eqref{eq:ILSE}, there exits a vector $\xi_{2}\in \R^{s}$ such that
$$
A^{\top }\Sigma_{pq}(b-Ax_{\ILSE})=B^{\top }\xi_{2},\quad Bx_{\ILSE}=d .
$$
Let $r_{1}=B^{\top }\xi_{2}-A^{\top }\Sigma_{pq}r_y$, and $r_{2}=d-By$, then 
$$
r_{1}=A^{\top }\Sigma_{pq}A(y-x_{\ILSE}),\quad r_{2}=-B(y-x_{\ILSE}).$$
Then
\[\left\|\begin{bmatrix}A^{\top }\Sigma_{pq}A\\B\end{bmatrix}\right\|_2 \|y-x_{\ILSE}\|_2  \geq \left\|\begin{bmatrix}r_{1}\\r_{2}\end{bmatrix}\right\|_2                        \geq\left\|\begin{bmatrix}B\xi_{1}-A^{\top }\Sigma_{pq}r_y \cr d-By\end{bmatrix}\right\|_2,\]
which completes the proof of this theorem.
\end{proof}\qed

Next, we analyze $\tau_o$ in \eqref{eq:qq}. From the definition of $\tau_0$, we have the following result.
\begin{theorem}
	With the notations above, we have $\tau_{0}=\max  \left  \{\theta_3,\, \alpha^{-1} \right\}$, where $$\alpha=\sigma_{\min  }\left(\begin{bmatrix}I_{n}\otimes (r_y^{\top }\Sigma_{pq})-A^{\top }\Sigma_{pq}(y^{\top }\otimes I_{m})&\theta_{1}^{-1}A^\top \Sigma_{pq} \end{bmatrix}\right).$$
\end{theorem}

The next theorem gives a lower bound of $\alpha$ and thus an upper bound of $\tau_0$.

\begin{theorem}
	If $r_y \neq 0$, then $\alpha \geq  \dfrac{\|r_y\|_2}{\sqrt{1+\theta_1^{2}\|y\|_2^2} }.
$
\end{theorem}
\begin{proof} Noting $\Sigma_{pq}^2=I_m$, since
\begin{align*}
	&\begin{bmatrix}I_{n}\otimes (r_y^{\top }\Sigma_{pq})-A^{\top }\Sigma_{pq}(y^{\top }\otimes I_{m})&\theta_{1}^{-1}A^\top \Sigma_{pq} \end{bmatrix}\,\begin{bmatrix}I_{n}\otimes (r_y^{\top }\Sigma_{pq})-A^{\top }\Sigma_{pq}(y^{\top }\otimes I_{m})&\theta_{1}^{-1}A^\top \Sigma_{pq} \end{bmatrix}^\top \\
	&=(\theta_1^{-2}+\|y\|_2^2)\left\{(A-r_yy_0)^\top (A-r_y y_0)+\frac{\|r_y\|_2^2}{\theta_1^{-2}+\|y\|_2^2} I_n- (r_yy_0)^\top (r_y y_0)\right\}\\
	&\geq (\theta_1^{-2}+\|y\|_2^2)\left\{\frac{\|r_y\|_2^2}{\theta_1^{-2}+\|y\|_2^2} I_n- (r_yy_0)^\top (r_y y_0)\right\},
\end{align*}
(here for two symmetric semi-positive matrix $M$ and $N$, $M\geq N$ means that $M-N$ is still semi positive), where $y_0=[1/(\theta_1^{-2}+\|y\|_2^2)] y$, we have
$$
\alpha^2 \geq  (\theta_1^{-2}+\|y\|_2^2)\left\{\frac{\|r_y\|_2^2}{\theta_1^{-2}+\|y\|_2^2} - \|r_y\|_2^2\|y_0\|_2^2 \right\}=\frac{\|r_y\|_2^2}{1+\theta_1^{2}\|y\|_2^2}.
$$
	
\end{proof}
\qed 

\section{Numerical examples}\label{sec:nume ex}

In this section we will test the effectiveness of the linearization estimate $\rho $ for the normwise backward error $\mu_{\ILSE}$ of ILSE \eqref{eq:ILSE}. All the
computations are carried out using \textsc{Matlab} 8.1 with the machine precision
$\epsilon=2.2 \times 10^{-16}$.

We adopt the method in \cite{MastronardiVanDooren2014BIT} to construct the data.  Let the matrix $A$, given $\kappa_A$, be generated as $A=QDU$, where $Q\in \R^{m\times m}$ is a $\Sigma_{pq}$-orthogonal matrix, i.e., such that $Q^\top \Sigma_{pq}Q=\Sigma_{pq}$, $D\in \R^{m\times n} $ is a diagonal matrix with decreasing diagonal values geometrically distributed between 1 and $\kappa_A$, and $U\in \R^{n\times n}$ is a random orthogonal matrix generated by the function ${\tt gallery} (\rq {\tt qumlt} \rq, n)$. Furthermore, $A$ is normalized such that $\|A\|_2=1$. The matrix $B\in \R^{s \times n}$, given its condition number $\kappa_B$, is formed by using \textsc{Matlab} routine $B={\tt gallery} (\rq {\tt randsvd} \rq,\, [s, \,n],\, \kappa_B)$ with $\|B\|_2=1$ and its singular values are geometrically distributed between $1$ and $1/\kappa_B$. We construct the random vectors $b$ and $d$ which are satisfied the standard Gaussian distribution for ILSE \eqref{eq:ILSE}.  For all the experiments, we choose $n = 50, \, s = 20,\, p = 60$, $q = 40$. For each generated data, we compute the solution via the augmented system \eqref{eq:aug}.  For the perturbations, we generate them as
\begin{equation*}\label{eq:pert}
\Delta A=\varepsilon \cdot \Delta A_{1},\quad \Delta B=\varepsilon \cdot \Delta B_{1}, \quad \Delta b=\varepsilon \cdot \Delta b_{1}\cdot  \|b\|_2, \quad \Delta d=\varepsilon \cdot \Delta d_{1}\cdot  \|d\|_2 
\end{equation*}
where each components of $\Delta A_1 \in \R^{m\times n}$, $\Delta B_1 \in \R^{s \times n}$, $\Delta b_1\in \R^{m}$ and $\Delta d_1\in \R^{s}$  satisfy the  standard Gaussian distribution. Let the computed solution $y$ be computed via solving the corresponding augmented system to the following perturbed ILSE problem
\begin{equation*}  
\min\left((b+\Delta b)-(A+\Delta A)y \right)^{\top}\Sigma_{pq}\left((b+\Delta b)-(A+\Delta A)y \right), \quad \mbox{ subject to }\quad (B+\Delta B) {y}=d+\Delta d.
\end{equation*}

For the computed solution $y$, its normwise backward error $\mu_\ILSE$  is defined by \eqref{eq:back}, and its linearization estimate $\rho$ for the normwise backward error $\mu_\ILSE$ is given by \eqref{eq:linear}. Because there is no explicit expression for $\rho$, we use $\rho(\xi_1)$ where $xi_1$ is given by \eqref{eq:xi1} to approximate $\mu_\ILSE$.  We always use the common choice  $\theta_1=\theta_2=\theta_3=1$ in \eqref{eq:back}. There is no explicit expression for the normwise backward error $\mu_\ILSE$. Since the perturbations $\Delta A$,  $\Delta B$, $\Delta b$ and $\Delta d$ are known in advance, we can calculate the following quantity $\mu_1$ to approximate $\mu_\ILSE$:
$$
 \mu_{1}= \left\|\begin{bmatrix}
\Delta A & \Delta  b\\
\Delta  B& \Delta d
\end{bmatrix}\right\|_{F} ,
 $$
 and compare $\mu_1$ with the linearization estimate $\rho(\xi_1)$ to show the effectiveness of $\rho(\xi_1)$. From the definition of the normwise backward error $\mu$ defined in \eqref{eq:back}, it is easy to see that $\mu \leq \mu_1$. Note that $\mu$ may be much smaller that $\mu_1$ because $\mu$ is the smallest perturbation magnitude over the set of all perturbations $S_{\ILSE}$. We test different choices of the perturbations magnitude $\varepsilon$  and the parameter $\kappa_A,\, \kappa_B$. We report the numerical values of $\mu_1$, $\rho(\xi_1)$ and the residual norms $\gamma$ and $\bar \gamma$ corresponding tho the original and perturbed augmented system in Table \ref{tab:b1}.
 
\begin{table}
  \caption{Comparisons between $\mu_1$ and $\rho(\xi_1)$.}
  \label{tab:b1}
  \centering
   \begin{tabular}{ccccccc}
   \hline
    $\varepsilon$& $ \kappa_A $ & $\kappa_B$ & $ \gamma$ & $ \bar \gamma$  & $ \mu_{1} $ & $ \rho(\xi_1)$\\
      \hline
 $10^{-6}$     &3.92e+01 & 1.00e+02 & 2.96e-08 & 3.28e-09 & 1.31e-04 & 3.68e-06   \\
&4.51e+02 & 1.00e+02 & 3.62e-08 & 4.64e-10 & 1.30e-04 & 4.79e-06   \\
&4.69e+04 & 1.00e+02 & 1.19e-07 & 5.35e-10 & 1.20e-04 & 3.53e-06   \\
&3.44e+08 & 1.00e+02 & 8.98e-07 & 6.04e-10 & 1.19e-04 & 8.70e-06   \\
&4.21e+01 & 1.00e+04 & 1.65e-08 & 1.76e-10 & 1.19e-04 & 5.56e-06   \\
&3.24e+02 & 1.00e+04 & 2.15e-08 & 3.65e-10 & 1.39e-04 & 5.03e-06   \\
&1.71e+04 & 1.00e+04 & 6.22e-08 & 4.02e-09 & 1.25e-04 & 5.08e-06   \\
&3.54e+08 & 1.00e+04 & 6.18e-07 & 6.23e-10 & 1.35e-04 & 5.80e-06   \\
&4.86e+01 & 1.00e+06 & 2.71e-08 & 3.78e-10 & 1.28e-04 & 4.51e-06   \\
&3.45e+02 & 1.00e+06 & 2.34e-08 & 4.89e-10 & 1.30e-04 & 4.25e-06   \\
&3.09e+04 & 1.00e+06 & 7.65e-08 & 5.27e-10 & 1.42e-04 & 4.58e-06   \\
&4.31e+08 & 1.00e+06 & 5.44e-07 & 1.22e-09 & 1.25e-04 & 4.15e-06   \\
&4.71e+01 & 1.00e+08 & 3.25e-08 & 2.12e-09 & 1.25e-04 & 3.56e-06   \\
&2.41e+02 & 1.00e+08 & 2.72e-08 & 2.15e-09 & 1.25e-04 & 4.54e-06   \\
&3.24e+04 & 1.00e+08 & 6.27e-08 & 1.93e-10 & 1.43e-04 & 3.44e-06   \\
&1.40e+08 & 1.00e+08 & 6.06e-07 & 8.98e-11 & 1.28e-04 & 5.63e-05 \\ 
\hline
$10^{-12}$ &  4.88e+01 & 1.00e+02 & 2.77e-08 & 2.60e-08 & 1.32e-10 & 7.90e-12   \\
&3.05e+02 & 1.00e+02 & 2.99e-08 & 2.53e-08 & 1.28e-10 & 8.41e-12   \\
&3.35e+04 & 1.00e+02 & 6.15e-08 & 4.90e-08 & 1.29e-10 & 8.52e-12   \\
&4.55e+08 & 1.00e+02 & 4.68e-07 & 4.65e-07 & 1.24e-10 & 3.39e-11   \\
&3.08e+01 & 1.00e+04 & 1.32e-08 & 1.48e-08 & 1.22e-10 & 7.98e-12   \\
&5.00e+02 & 1.00e+04 & 2.96e-08 & 3.33e-08 & 1.35e-10 & 8.69e-12   \\
&1.15e+04 & 1.00e+04 & 5.37e-08 & 8.57e-08 & 1.22e-10 & 1.09e-11   \\
&2.72e+08 & 1.00e+04 & 7.73e-07 & 6.62e-07 & 1.29e-10 & 9.02e-11   \\
&3.18e+01 & 1.00e+06 & 3.23e-08 & 3.04e-08 & 1.29e-10 & 7.88e-12   \\
&3.38e+02 & 1.00e+06 & 2.14e-08 & 1.91e-08 & 1.29e-10 & 1.06e-11  \\ 
&2.64e+04 & 1.00e+06 & 1.02e-07 & 8.75e-08 & 1.33e-10 & 1.26e-11  \\ 
&3.04e+08 & 1.00e+06 & 8.11e-07 & 6.82e-07 & 1.16e-10 & 1.09e-10  \\ 
&1.92e+01 & 1.00e+08 & 2.04e-08 & 2.42e-08 & 1.25e-10 & 4.20e-10  \\ 
&1.14e+02 & 1.00e+08 & 3.20e-08 & 3.37e-08 & 1.35e-10 & 7.51e-11  \\ 
&3.75e+04 & 1.00e+08 & 4.60e-08 & 6.50e-08 & 1.20e-10 & 2.73e-10  \\ 
&2.73e+08 & 1.00e+08 & 6.86e-07 & 4.71e-07 & 1.29e-10 & 1.25e-10   \\
      \hline
    \end{tabular}
\end{table}

From Table \ref{tab:b1}, it is observed that the residual norms $\gamma$ and $\bar \gamma $ are always small regardless of different choices of $\varepsilon$, $ \kappa_A $ and $\kappa_B$. Thus the solutions $x$ and $y$ are acceptable in the sense of the residual norms for the augmented system. The differences between $\mu_1$ and $\rho(\xi_1)$ are not too big. Most values of $\mu_1$ are one hundredfold of the corresponding of $\rho(\xi_1)$. However, we cannot conclude that $\rho(\xi_1)$ gives a bad estimation for $\mu_\ILSE$ because $\mu_\ILSE$ is the smallest perturbation magnitude to let the computed solution $y$ be the exact solution of the perturbed ILSE mathematically. The values of $\rho(\xi_1)$ are coincided with the perturbation  magnitude $\varepsilon$, which indicates the linearized estimation $\rho(\xi_1)$ is effective. 

\section{Concluding Remarks}\label{sec:con}

In this paper we studied the linearization estimate for the normwise backward error of  the equality constrained indefinite least squares problem. The explicit sub-optimal  linearization estimate is given. We tested the derived sub-optimal linearization estimate through numerical examples, which showed that it is reliable and effective.




\begin{thebibliography}{10}

\bibitem{5.0}
Adam Bojanczyk, Nicholas~J. Higham, and Harikrishna Patel.
\newblock Solving the indefinite least squares problem by hyperbolic {QR}
  factorization.
\newblock {\em SIAM J. Matrix Anal. Appl.}, 24(4):914--931 (electronic), 2003.

\bibitem{Higham2003ILSE}
Adam Bojanczyk, Nicholas{\thinspace}J. Higham, and Harikrishna Patel.
\newblock The equality constrained indefinite least squares problem: Theory and
  algorithms.
\newblock {\em BIT Numerical Mathematics}, 43(3):505--517, 2003.

\bibitem{1.0}
S.~Chandrasekaran, M.~Gu, and A.~H. Sayed.
\newblock A stable and efficient algorithm for the indefinite linear
  least-squares problem.
\newblock {\em SIAM J. Matrix Anal. Appl.}, 20(2):354--362, 1999.

\bibitem{Chang}
X.-W. Chang and D.~Titley-Peloquin.
\newblock Backward perturbation analysis for scaled total least-squares
  problems.
\newblock {\em Numer. Linear Algebra Appl.}, 16(8):627--648, 2009.

\bibitem{Cox}
Anthony~J. Cox and Nicholas~J. Higham.
\newblock Backward error bounds for constrained least squares problems.
\newblock {\em BIT}, 39(2):210--227, 1999.

\bibitem{Grcar03optimalsensitivity}
Joseph~F. Grcar.
\newblock Optimal sensitivity analysis of linear least squares.
\newblock Technical report, 2003.

\bibitem{Grcar2007}
Joseph~F. Grcar, Michael~A. Saunders, and Zheng Su.
\newblock Estimates of optimal backward perturbations for linear least squares
  problems.
\newblock Technical report, 2007.

\bibitem{Hassibi96linearestimation}
Babak Hassibi, Ali~H. Sayed, and Thomas Kailath.
\newblock Linear estimation in {K}rein spaces - {P}art {I}: {T}heory.
\newblock {\em IEEE Transactions on Autmatic Control}, 3(2):18--33, 1996.

\bibitem{HighamHigham1992BackCondStr}
Desmond~J. Higham and Nicholas~J. Higham.
\newblock Backward error and condition of structured linear systems.
\newblock {\em SIAM J. Matrix Anal. Appl.}, 13(1):162--175, 1992.

\bibitem{Higham}
Nicholas~J. Higham.
\newblock Computing error bounds for regression problems.
\newblock In {\em Statistical analysis of measurement error models and
  applications ({A}rcata, {CA}, 1989)}, volume 112 of {\em Contemp. Math.},
  pages 195--208. Amer. Math. Soc., Providence, RI, 1990.

\bibitem{Higham2002Book}
Nicholas~J. Higham.
\newblock {\em Accuracy and stability of numerical algorithms}.
\newblock SIAM, Philadelphia, PA, second edition, 2002.

\bibitem{Karlson}
Rune Karlson and Bertil Wald{\'e}n.
\newblock Estimation of optimal backward perturbation bounds for the linear
  least squares problem.
\newblock {\em BIT}, 37(4):862--869, 1997.

\bibitem{Liu2010ILSEsolver}
Qiaohua Liu, Baozhen Pan, and Qian Wang.
\newblock The hyperbolic elimination method for solving the equality
  constrained indefinite least squares problem.
\newblock {\em Int. J. Comput. Math.}, 87(13):2953--2966, 2010.

\bibitem{LiuWang2010ILSE}
Qiaohua Liu and Minghui Wang.
\newblock Algebraic properties and perturbation results for the indefinite
  least squares problem with equality constraints.
\newblock {\em Int. J. Comput. Math.}, 87(1-3):425--434, 2010.

\bibitem{LiuXinguoNLAA2012}
Xin-Guo Liu and Na~Zhao.
\newblock Linearization estimates of the backward errors for least squares
  problems.
\newblock {\em Numer. Linear Algebra Appl.}, 19(6):954--969, 2012.

\bibitem{Malyshev}
A.~N. Malyshev.
\newblock Optimal backward perturbation bounds for the {LSS} problem.
\newblock {\em BIT}, 41(2):430--432, 2001.

\bibitem{MastronardiVanDooren2014BIT}
Nicola Mastronardi and Paul Van~Dooren.
\newblock An algorithm for solving the indefinite least squares problem with
  equality constraints.
\newblock {\em BIT}, 54(1):201--218, 2014.

\bibitem{Mastronardi2015IMA}
Nicola Mastronardi and Paul Van~Dooren.
\newblock A structurally backward stable algorithm for solving the indefinite
  least squares problem with equality constraints.
\newblock {\em IMA J. Numer. Anal.}, 35(1):107--132, 2015.

\bibitem{Sayed96inertiaproperties}
Ali~H. Sayed, Babak Hassibi, and Thomas Kailath.
\newblock Inertia properties of indefinite quadratic forms.
\newblock {\em IEEE Signal Process. Lett}, 3(2):57--59, 1996.

\bibitem{Stewart}
G.~W. Stewart.
\newblock Research, development, and {LINPACK}.
\newblock In {\em Mathematical software, {III} ({P}roc. {S}ympos., {M}ath.
  {R}es. {C}enter, {U}niv. {W}isconsin, {M}adison, {W}is., 1977)}, pages 1--14.
  Publ. Math. Res. Center, No. 39. Academic Press, New York, 1977.

\bibitem{30.0}
Sabine Van~Huffel and Joos Vandewalle.
\newblock {\em The total least squares problem}, volume~9 of {\em Frontiers in
  Applied Mathematics}.
\newblock Society for Industrial and Applied Mathematics (SIAM), Philadelphia,
  PA, 1991.
\newblock Computational aspects and analysis, With a foreword by Gene H. Golub.

\bibitem{Walden}
Bertil Wald{\'e}n, Rune Karlson, and Ji~Guang Sun.
\newblock Optimal backward perturbation bounds for the linear least squares
  problem.
\newblock {\em Numer. Linear Algebra Appl.}, 2(3):271--286, 1995.

\end{thebibliography}
\end{document}